\documentclass{extarticle}
\usepackage{amsfonts,amssymb,amsmath,amsthm}
\usepackage{url}
\usepackage{enumerate}
\usepackage{color}
\usepackage{nameref}
\usepackage[dvipsnames]{xcolor}
\usepackage[colorlinks=true,linkcolor=RoyalBlue,citecolor=PineGreen,urlcolor=blue]{hyperref}
\usepackage{enumitem}
\usepackage[margin=3cm,letterpaper]{geometry}
\usepackage{bm}
\usepackage[numbers,sort]{natbib}
\newtheorem{lemma}{Lemma}[section]
\newtheorem{step}{Step}[lemma]

\newtheorem{theorem}{Theorem}[section]

\newtheorem{claim}{Claim}[theorem]

\newtheorem{problem}{Problem}

\theoremstyle{definition}
\newtheorem{example}{Example}[section]
\newtheorem{remark}{Remark}[section]
\newtheorem{definition}{Definition}[section]
\newtheorem{notation}{Notation}[section]

\DeclareMathOperator{\trdeg}{trdeg}

\DeclareMathOperator{\rank}{rank}

\DeclareMathOperator{\DS}{\Delta{\hbox{-}}\Sigma}

\makeatletter
\newcommand{\myitem}[1]{%
\item[#1]\protected@edef\@currentlabel{#1}%
}
\makeatother

\def\subjclass#1{{\renewcommand{\thefootnote}{}%
\footnote{\emph{Mathematics Subject Classification (2010):} #1}}}

\title{A Primitive Element Theorem for Fields with\\ Commuting Derivations and Automorphisms}

\author{Gleb Pogudin\footnote{{pogudin.gleb@gmail.com}, Courant Institute of Mathematical Sciences, New York University, New York,  USA}}

\date{}

\begin{document}

\maketitle

\begin{abstract}
    We establish a Primitive Element Theorem for fields equipped with several commuting operators such that each of the operators is either a derivation or an automorphism.
    More precisely, we show that for every extension $F \subset E$ of such fields of zero characteristic such that 
    \begin{itemize}
        \item $E$ is generated over $F$ by finitely many elements using the field operations and the operators,
        \item every element of $E$ satisfies a nontrivial equation with coefficient in $F$ involving the field operations and the operators,
        \item the action of the operators on $E$ is irredundant
    \end{itemize}
    there exists an element $a \in E$ such that $E$ is generated over $F$ by $a$ using the field operations and the operators.
    
    This result generalizes the Primitive Element Theorems by Kolchin and Cohn in two directions simultaneously: we allow any numbers of derivations and automorphisms and do not impose any restrictions on the base field $F$.
    
\end{abstract}


\section{Introduction}

\subsection{Overview and prior results}

 The Primitive Element Theorem~\cite[\S 40]{vdW} is a fundamental result in the field theory that says that for every separable finitely generated algebraic extension of fields $F \subset E$, there exists $a \in E$ such that $E = F(a)$.
Apart from its theoretical importance, it is also one of the main tools for computing with algebraic extensions~\cite[\S 5]{Loos}.

\paragraph{Fields with commuting derivations.} Consider an extension $\mathbb{C} \subset \mathbb{C}(x, e^x)$.
The field $\mathbb{C}(x, e^x)$ cannot be generated over $\mathbb{C}$ by one element because the functions $x$ and $e^x$ are algebraically independent.
However, the formulas
\[
x = \alpha - \alpha' + 1 \quad \text{ and } \quad e^x = \alpha' - 1, \text{ where } \alpha := x + e^x
\]
imply that $\mathbb{C}(x, e^x)$ is generated over $\mathbb{C}$ by $\alpha$ if we allow taking derivatives as well as the field operations.

Fields equipped with commuting derivations are central objects in the algebraic studies of differential equations initiated by Ritt~\cite{Ritt} and Kolchin~\citep{KolchinBook}.
In this setting, Kolchin proved the following analogue of the Primitive Element Theorem:

\begin{theorem}[Kolchin, {{\cite[\S 4]{primitiveKolchin}}}]\label{thm:Kolchin}
  Let $F \subset E$ be an extension of fields of zero characteristic and $E$ be equipped with commuting derivations $\delta_1, \ldots, \delta_s$ such that $F$ is closed under the derivations. 
  Assume that
  \begin{enumerate}[itemsep=1pt, topsep=1pt, label=(\arabic*)]
      \item there exist $a_1, \ldots, a_n \in E$ such that
      \[
      E = F\bigl( \delta_1^{\alpha_1}\ldots \delta_s^{\alpha_s} a_j \mid 1 \leqslant j \leqslant n,\; \alpha_1, \ldots, \alpha_s \in \mathbb{Z}_{\geqslant 0} \bigr).
      \]
      \item for every $1 \leqslant j \leqslant n$, $a_j$ satisfies a nonzero polynomial PDE over $F$, that is, the elements of $\{\delta_1^{\alpha_1}\ldots \delta_s^{\alpha_s} a_j \mid \alpha_1, \ldots, \alpha_s \in \mathbb{Z}_{\geqslant 0}\}$ are algebraically dependent over $F$;
      \item\label{cond:Kolchin} there exist $b_1, \ldots, b_s \in F$ with nondegenerate Jacobian, that is, $\det (\delta_i b_j)_{i, j = 1}^s \neq 0$.
  \end{enumerate}
  Then there exists $a \in E$ such that $E = F(\delta_1^{\alpha_1}\ldots \delta_s^{\alpha_s} a \mid \alpha_1, \ldots, \alpha_s \in \mathbb{Z}_{\geqslant 0})$.
\end{theorem}

Theorem~\ref{thm:Kolchin} and its improvements~\cite{Seidenberg,primitive} have been used, for example, in algorithms and effective bounds for differential-algebraic equations~\cite{Luroth,CluzeauHubert, CluzeauHubert2, MOT16,FreitagLi},  Galois theory of differential and difference equations \cite{HardouinSinger,BP2016_model}, model theory of differential fields~\cite{Wood74, MarkerModel}, control theory~\cite{BellLu,Fliess}, and for connecting algebraic and analytic approaches to differential-algebraic equations~\cite{SeidenbergEmbedding, SeidenbergEmbedding2, S78}.

In our earlier paper~\cite{primitive}, for the case of one derivation ($m = 1$), condition~\ref{cond:Kolchin} has been relaxed to the condition that $E$ contains a nonconstant, i.e. that the derivation is not zero.
Unlike Theorem~\ref{thm:Kolchin}, this refined statement is applicable, for example, to the extension $\mathbb{C} \subset \mathbb{C}(x, e^x)$ discussed above and to extensions of the form $\mathbb{C} \subset \mathbb{C}(X)$, where $\mathbb{C}$ is the field of rational functions on an irreducible algebraic variety $X$ and the derivation on $\mathbb{C}(X)$ is induced by a vector field on $X$.
It is an open problem whether such relaxation is true for the case of several commuting derivations.

\begin{problem}\label{prob:differential}
    Prove or disprove that Theorem~\ref{thm:Kolchin} is still true if condition~\ref{cond:Kolchin} is replaced with \begin{itemize}[itemsep=1pt, topsep=1pt]
      \item[(3')] there exist $b_1, \ldots, b_s \in E$ (not $F$) with nondegenerate Jacobian.
  \end{itemize}  
\end{problem}

\paragraph{Fields with an automorphism.} Consider the extension $\mathbb{C} \subset \mathbb{C}(x, \Gamma(x))$, where $\Gamma(x)$ is the gamma function.
Since $x$ and $\Gamma(x)$ are algebraically independent, $\mathbb{C}(x, \Gamma(x))$ cannot be generated over $\mathbb{C}$ by one element.
The difference equation for the gamma function $\Gamma(x + 1) = x\Gamma(x)$ implies that the shift of the argument $f(x) \mapsto f(x + 1)$ induces an automorphism on $\mathbb{C}(x, \Gamma(x))$.
Then the formula $x = \frac{\Gamma(x + 1)}{\Gamma(x)}$ implies that $\mathbb{C}(x, \Gamma(x))$ is generated over $\mathbb{C}$ by $\Gamma(x)$ if we allow integer shifts of the argument as well as the field operations.

Algebraic theory of nonlinear difference equations founded by Ritt and Cohn extensively uses fields equipped with an automorphism.
Cohn established the following version of the Primitive Element Theorem (to keep the presentation simple, we restrict ourselves to the zero characteristic case):

\begin{theorem}[Cohn~{\cite[p. 203, Theorem~III]{Cohn}}]\label{thm:Cohn}
  Let $F \subset E$ be an extension of fields of zero characteristic and $\sigma\colon E \to E$ be an automorphism such that $\sigma(F) \subset F$. Assume that
  \begin{enumerate}[itemsep=1pt, topsep=1pt, label=(\arabic*)]
      \item there exist $a_1, \ldots, a_n \in E$ such that $E = F(\sigma^i (a_j) \;|\; i \in \mathbb{Z},\; 1 \leqslant j \leqslant n)$;
      \item for every $1 \leqslant j\leqslant n$, $a_j$ satisfies a nonzero difference-algebraic equation over $F$, that is, the elements of $\{\sigma^i (a_j) \;|\; i \in \mathbb{Z}\}$ are algebraically dependent over $F$;
      \item\label{cond:Cohn_nonperiodic} $\sigma$ has infinite order on $F$, that is, there is no integer $k \geqslant 1$ such that, for every $a \in F$, $\sigma^k(a) = a$.
  \end{enumerate}
  Then there exists $a \in E$ such that $E = F(\sigma^j(a) \mid j \in \mathbb{Z})$.
\end{theorem}

Theorem~\ref{thm:Cohn} has been used, for example, in model theory of fields with an automorphism~\cite{CH99} and for proving approximation theorem for difference equations~\cite{Belair}.
However, Theorem~\ref{thm:Cohn} is not applicable to our example $\mathbb{C} \subset \mathbb{C}(x, \Gamma(x))$ because the automorphism acts trivially on the base field $\mathbb{C}$.

\begin{problem}\label{prob:difference}
  Prove or disprove that Theorem~\ref{thm:Cohn} is still true if condition~\ref{cond:Cohn_nonperiodic} is replaced with \begin{itemize}[itemsep=1pt, topsep=1pt]
      \item[(3')] $\sigma$ has inifinite order on $E$ (not $F$).
  \end{itemize}
\end{problem}

Positive solution (i.e., ``prove'') to Problem~\ref{prob:difference} would have, for example, the following application.
Let $\mathbb{C} \subset \mathbb{C}(X)$, where $\mathbb{C}(X)$ is the field of rational functions on an irreducible algebraic variety $X$.
Consider an automophism $\alpha$ of $X$ of infinite order. 
The dual $\alpha^\sharp \colon \mathbb{C}(X) \to \mathbb{C}(X)$ is an automorphism of $\mathbb{C}(X)$ of infinite order.
Then the positive solution to Problem~\ref{prob:difference} would imply that there exists $f \in \mathbb{C}(X)$ such that the orbit of $f$ under $\alpha^\sharp$ generates $\mathbb{C}(X)$ over~$\mathbb{C}$.
In particular, our main result, Theorem~\ref{thm:main}, implies the existence of such $f$.

%%%%%%%%%%%%%%%%%%

\paragraph{More general cases.}
Although fields with several commuting automorphisms and fields with several derivations and automorphisms commuting with each other have been studied from the standpoints of algebra~\cite{Cohn, LevinBook}, model theory~\cite{Sanchez16,BM07,C15}, and symbolic computation~\cite{Levin, GHYZ09}, we are not aware of any analogues of the Primitive Element Theorem for such fields.

\begin{problem}\label{prob:mixed}
  Derive an analogue of the Primitive Element Theorem for fields with several commuting derivations and automorphisms.
\end{problem}

Another common generalization of fields equipped with a derivations and fields equipped with an automorphism is the theory of fields with free operators introduced in~\cite{MS10,MS14} (see also~\cite{K13,BHMP17}).
We are not aware of any analogues of the Primitive Element Theorem for such fields.

%%%%%%%%%%%%%%%%%%%

\subsection{Our contribution}

Our main result, Theorem~\ref{thm:main}, generalizes Theorems~\ref{thm:Kolchin} and~\ref{thm:Cohn} in two directions simultaneously
\begin{itemize}
    \item We establish a Primitive Element Theorem for fields equipped with $s$ derivations and $t$ automorphisms such that all these operators commute.
    This solves Problem~\ref{prob:mixed}.
    
    \item We remove all the assumptions on the base field $F$ other than being closed under the operators.
    Instead of this, we require the derivations and automorphisms to be ``independent'' on the bigger field $E$.
    In particular, this solves Problems~\ref{prob:differential} and~\ref{prob:difference} (see Remark~\ref{rem:CohnKolchin}).
    We show (see Example~\ref{ex:cannot_drop}) that the condition on $E$ cannot be removed.
\end{itemize}

%%%%%%%%%

\subsection{Outline of the approach}

One of the main challenges in proving such a general version of the Primitive Element Theorem is to find an appropriate general form of a primitive element.
    Proofs of Kolchin~\cite{primitiveKolchin} and Cohn~\cite{Cohn} follow the same strategy as the standard proof of the algebraic Primitive Element Theorem, namely they construct a primitive element of an extension $F \subset E$ as a \emph{$F$-linear combination} of the original generators.

However, Examples~\ref{ex:not_linear_comb_differential} and~\ref{ex:not_linear_comb_difference} show that if $F$ is a constant field, then it can happen that none of the linear combinations of the original generators is a primitive element even in the case of only one operator.
To strengthen Kolchin's theorem (Theorem~\ref{thm:Kolchin}) in the case of one derivation, in our earlier paper~\cite{primitive}, we used an involved two-stage construction of a primitive element as a \emph{polynomial in the original generators}.

In this paper, we approach finding a primitive element from the perspective of \emph{the Taylor series expansion}.
More precisely, a primitive element is constructed as a truncated multivariate Taylor series in the original generators and their derivatives (see formulas~\eqref{eq:primitive_differential} and~\eqref{eq:primitive_general}).
The order of truncation is derived based on the Kolchin polynomial of the extension.
Formally, a truncated Taylor series is simply a polynomial written in a special form with factorials in the denominators, so the search space for a primitive element is almost the same as in~\cite{primitive}.
However, this representation turns out to be a key to interpreting a polynomial system that relates the original generators and a potential primitive element in terms of solutions some special system of linear PDEs.
This interpretation allows us to show that the original generators can be expressed in terms of the potential primitive element (see Lemmas~\ref{lem:core_differential} and~\ref{lem:core}), so this element is indeed primitive.
As was pointed out to me by Jonathan Kirby, one can view the difference between the representations of a primitive element in~\cite{primitive} and in the present paper as the difference between usual generating series and exponential generating series.

We refer a reader who wants to see more details but does not want to read the whole proof to Section~\ref{sec:differential_proof}.
This section contains a proof of the main theorem for differential fields (i.e., fields equipped with a derivation).
It is much shorter than the main proof but  demonstrates  some of the key techniques we developed.

%%%%%%%%%%%%%%%%%%%

\subsection{Structure of the paper}

The rest of the paper is organized as follows. Section~\ref{sec:defs_mainresult} contains definitions used in the statement of the main result and the main result, Theorem~\ref{thm:main}.
Section~\ref{sec:examples} contains examples that illustrate the main result.
Section~\ref{sec:def_proofs} contains the definitions and notation used in the proofs.
Section~\ref{sec:differential_proof} contains the proof of the main result in the special case of fields with a derivation. This proof demonstrates some of the key ingredients of the proof of Theorem~\ref{thm:main} in a simpler setting allowing the reader to understand of the general approach without going into technical details.
Section~\ref{sec:proof_main} contains the proof of the main result.
For the convenience of the reader, the corresponding lemmas in Sections~\ref{sec:differential_proof} and~\ref{sec:proof_main} are cross-referenced.

%%%%%%%%%%%%%%%%%%%%%%%%

\section{Definitions and the Main Result}
\label{sec:defs_mainresult}

All fields are assumed to be of zero characteristic.

\begin{definition}[$\DS$-rings]
  Let $\Delta = \{\delta_1, \ldots, \delta_s\}$ and $\Sigma = \{\sigma_1, \ldots, \sigma_t\}$ be finite sets of symbols.
  We say that a commutative ring $R$ is a \emph{$\DS$-ring} if 
  \begin{enumerate}[itemsep=1pt, topsep=1pt]
      \item $\delta_1, \ldots, \delta_s$ act on $R$ as derivations, that is, $\delta_i(a + b) = \delta_i(a) + \delta_i(b)$ and $\delta_i(ab) = \delta_i(a)b + a\delta_i(b)$ for every $1 \leqslant i \leqslant s$ and $a, b \in R$;
      \item $\sigma_1, \ldots, \sigma_t$ act on $R$ as automorphisms;
      \item every two operators in $\Delta \cup \Sigma$ commute.
  \end{enumerate}
  A $\DS$-ring that is a field is called \emph{$\DS$-field}.
\end{definition}

\begin{example}
  Natural examples of $\DS$-fields include the following.
  \begin{itemize}[itemsep=0pt, topsep=1pt]
      \item Let $E = \mathbb{C}(x)$, $\Delta = \{\partial\}$, and $\Sigma = \{\sigma\}$.
      We can define a structure of $\DS$-field on $E$ by
      \[
        \delta( f(x) ) = \frac{\partial}{\partial x}f(x),\; \text{ and }\; \sigma(f(x)) = f(x + 1) \quad \text{ for every } f(x) \in E.
      \]
      In the same way, one can define $n$ derivations and $n$ automorphisms on $\mathbb{C}(x_1, \ldots, x_n)$.
      \item Let $E$ be the field of meromorphic functions on $\mathbb{C}$, $\Delta = \{\delta\}$, and $\Sigma = \{\sigma_1, \sigma_2\}$.
      We can define a structure of $\DS$-field on  $E$ by
      \[
      \delta(f(z)) = f'(z), \;\; \sigma_1(f(z)) = f(z + 1),\;\; \sigma_2(f(z)) = f(z + i) \text{ for every } f \in E.
      \]
      \item Let $E$ be the field of meromorphic functions on $\mathbb{C}$, $\Delta = \{\delta\}$, and $\Sigma = \{\sigma\}$.
      For every nonzero $q \in \mathbb{C}$, we can define a structure of $\DS$-field on $E$ by
      \[
      \delta(f(z)) = zf'(z), \;\; \sigma(f(z)) = f(qz) \text{ for every } f \in E.
      \]
  \end{itemize}
\end{example}

\begin{definition}[Extension of $\DS$-fields]
  An extension of fields $F \subset E$ where $E$ and $F$ are $\DS$-fields is said to be \emph{an extension of $\DS$-fields} if the action of $\Delta\cup\Sigma$ on $F$ coincides with the restriction to $F$ of the action of $\Delta\cup\Sigma$ on $E$.
\end{definition}

\begin{notation}
  For every $\bm{\alpha} := (\alpha_1, \ldots, \alpha_s) \in \mathbb{Z}_{\geqslant 0}^s$ and every $\bm{\beta} := (\beta_1, \ldots, \beta_t) \in \mathbb{Z}^t$, we introduce
  \[
    \delta^{\bm{\alpha}} := \delta_1^{\alpha_1}\ldots \delta_s^{\alpha_s} \quad\text{ and }\quad \sigma^{\bm{\beta}} := \sigma_1^{\beta_1} \ldots \sigma_t^{\beta_t}.
  \]
\end{notation}

\begin{theorem}[Main Result]\label{thm:main}
  Let $F \subset E$ be an extension of $\DS$-fields such that 
  \begin{itemize}
      \myitem{(1)} there exist $a_1, \ldots, a_n \in E$ such that $E = F\bigl( \delta^{\bm{\alpha}} \sigma^{\bm{\beta}} (a_j) \mid 1 \leqslant j \leqslant n,\; \bm{\alpha} \in \mathbb{Z}_{\geqslant 0}^s,\; \bm{\beta} \in \mathbb{Z}^t\bigr)$;
      \myitem{(2)} for every $1 \leqslant j \leqslant n$, the elements of $\{ \delta^{\bm{\alpha}} \sigma^{\bm{\beta}}(a_j) \mid \bm{\alpha} \in \mathbb{Z}_{\geqslant 0}^s,\; \bm{\beta} \in \mathbb{Z}^t\}$ are algebraically dependent over $F$;
      \myitem{(3$\delta$)}\label{item:three_delta} $\delta_1, \ldots, \delta_s$ are linearly independent over $E$;
      \myitem{(3$\sigma$)}\label{item:three_sigma} $\sigma_1, \ldots, \sigma_t$ are multiplicatively independent over $\mathbb{Z}$, that is, for every $\bm{\beta}\in \mathbb{Z}^t$, $\sigma^{\bm{\beta}}|_E = \operatorname{id} \iff \bm{\beta} = \bm{0}$.
  \end{itemize}
  Then there exists $a \in E$ such that $E = F\bigl( \delta^{\bm{\alpha}} \sigma^{\bm{\beta}} (a) \mid \bm{\alpha} \in \mathbb{Z}_{\geqslant 0}^s,\; \bm{\beta} \in \mathbb{Z}^t\bigr)$.
\end{theorem}

\begin{remark}\label{rem:CohnKolchin}
  Setting $s = 0$ and $t = 1$, we obtain a statement stronger than Cohn's theorem~\cite[p. 203, Theorem~III]{Cohn} in the case of zero characteristic. 
  Lemma~\ref{lem:nondegenerate_field_has_nonzero_Jac} implies that the requirement on the base field $F$ in Kolchin's theorem~\cite[\S 4]{primitiveKolchin} is the same as the condition~\ref{item:three_delta} on $E$ in Theorem~\ref{thm:main}.
  Thus, Theorem~\ref{thm:main} strengthens Kolchin's theorem as well.
\end{remark}

%%%%%%%%%%%%%%%%%%%%%%%%

\section{Examples}
\label{sec:examples}

\begin{notation}
  Let $F \subset E$ be an extension of $\DS$-fields. 
  For $a_1, \ldots, a_n \in E$, we set 
  \[
      F\langle a_1, \ldots, a_n \rangle := F\bigl( \delta^{\bm{\alpha}}\sigma^{\bm{\beta}} (a_j) \mid 1 \leqslant j \leqslant n,\; \bm{\alpha} \in \mathbb{Z}_{\geqslant 0}^s,\; \bm{\beta} \in \mathbb{Z}^t  \bigr).
    \]
\end{notation}

\begin{example}\label{ex:jacobi}
  This example illustrates how Theorem~\ref{thm:main} can be applied to classical special functions.
  Let $\Delta = \{\partial_z, \partial_\tau\}$ and $\Sigma = \varnothing$.
  Let $\mathcal{M}(\mathbb{C}, \mathcal{H})$ denote the field of bivariate meromorphic functions on $\mathbb{C} \times \mathcal{H}$ in variables $z$ and $\tau$, where $\mathcal{H} = \{\tau \in \mathbb{C} \mid \operatorname{Im}(\tau) > 0\}$.
  We consider $\mathcal{M}(\mathbb{C}, \mathcal{H})$ as a $\Delta$-field by letting $\partial_z$ and $\partial_\tau$ act as the partial derivatives with respect to $z$ and $\tau$, respectively.
  Let
  \[
  \theta_1(z, \tau) := -i \sum\limits_{j = -\infty}^\infty (-1)^j e^{(j + 1/2)^2\pi i \tau} e^{(2j + 1)\pi i z}
  \]
  be one of the Jacobi theta functions~\cite[\S~10.7]{AAR99}.
  $\theta_1(z, \tau)$ satisfies the following form of the heat equation~\cite[p.~433]{theta}
  \[
    \partial_z^2 \theta_1(z, \tau) = 4\pi i \partial_\tau \theta_1(z, \tau).
  \]
  This implies that $\theta_1(z, \tau)$ as well as $\theta_1(2z, \tau)$ and $\theta_1(3z, \tau)$ are $\DS$-algebraic over $\mathbb{C}$.
  Thus, Theorem~\ref{thm:main} applied to the extension 
  \[
  F := \mathbb{C} \subset E :=  \mathbb{C}(\theta_1(z, \tau), \theta_1(2z, \tau), \theta_1(3z, \tau))
  \]
  implies that there exists a function $f(z, \tau) \in E$ such that $\theta_1(z, \tau), \theta_1(2z, \tau)$, and $\theta_1(3z, \tau)$ can be written as rational functions in $f$ and its partial derivatives.
  Note that the original Kolchin's theorem (Theorem~\ref{thm:Kolchin}) is not applicable to this extension.
\end{example}

Examples~\ref{ex:not_linear_comb_differential} and~\ref{ex:not_linear_comb_difference} show that it might be impossible to construct a primitive element of an extension as a linear combination of the original generators even in the case of one operator (see also~\cite[Remark~2]{primitive}).

\begin{example}\label{ex:not_linear_comb_differential}
  Let $\Delta = \{\delta\}$ and $\Sigma = \varnothing$.
  Field $E := \mathbb{C}(x, \ln x, \ln (1 - x))$ is a $\DS$-field with $\delta(f(x)) = f'(x)$ for every $f \in E$.
  Since $x = \frac{1}{(\ln x)'}$, $E = \mathbb{C}\langle \ln x, \ln(1 - x) \rangle$.
  Since each of $\ln x$ and $\ln(1 - x)$ satisfies an algebraic differential equation with constant coefficients, the extension 
  \[
  F := \mathbb{C} \subset E = \mathbb{C}\langle \ln x, \ln (1 - x)\rangle
  \]
  satisfies the conditions of Theorem~\ref{thm:main}.
  Consider arbitrary $\alpha, \beta \in F$ and set $a = \alpha\ln x + \beta\ln(1 - x)$.
  Since $a' \in \mathbb{C}(x)$, $F\langle a\rangle \subset F(a, x)$.
  The latter has the transcendence degree at most two over $F$ but Ostrowski's theorem~\citep{O20} implies that $\ln x, \ln (1 - x)$, and $x$ are algebraically independent, so $\trdeg_F E = 3$.
  Hence, $F\langle a \rangle \neq E$.
  Thus, such $a$ cannot be a primitive element of the extension $F \subset E$.
\end{example}

\begin{example}\label{ex:not_linear_comb_difference}
  Let $\zeta(z, s) = \sum\limits_{n = 0}^\infty \frac{1}{(z + n)^s}$ be the Hurwitz zeta function~\citep[\S~1.3]{AAR99}.
  Let $\mathcal{M}(\mathbb{C})$ be the field of meromorphic functions on $\mathbb{C}$.
  Let $E := \mathbb{C}(z, \zeta(z, 2), \zeta(1 - z, 2))$ be a subfield in $\mathcal{M}(\mathbb{C})$.
  We define a $\DS$-structure on $\mathcal{M}(\mathbb{C})$ with $\Delta = \varnothing$ and $\Sigma = \{\sigma\}$ by
  \[
  \sigma(f(z)) = f(z + 1) \text{ for every } f \in \mathcal{\mathbb{C}}.
  \]
  Since
  \begin{equation}\label{eq:zeta_diff}
    \sigma(\zeta(z, 2)) = \zeta(z, 2) - \frac{1}{z^2} \quad \text{ and } \quad \sigma(\zeta(1 - z, 2)) = \zeta(1 - z, 2) + \frac{1}{z^2},
  \end{equation}
  $E$ is a $\DS$-subfield of $\mathcal{M}(\mathbb{C})$, and $E = \mathbb{C} \langle \zeta(z, 2), \zeta(1 - z, 2) \rangle$.
  \eqref{eq:zeta_diff} implies that $\zeta(z, 2)$ and $\zeta(1 - z, 2)$ are $\DS$-algebraic over $\mathbb{C}$, so the extension
  \[
  F := \mathbb{C} \subset E = \mathbb{C}\langle \zeta(z, 2), \zeta(1 - z, 2)\rangle
  \]
  satisfies the conditions of Theorem~\ref{thm:main}.
  We claim that $\trdeg_F E = 3$.
  Assume the contrary, that is, $\zeta(z, 2), \zeta(1 - z, 2)$, and $z$ are algebraically dependent over $\mathbb{C}$.
  Let $\mathcal{C} := \{f \in \mathcal{M}(\mathbb{C}) \mid \sigma(f) = f\}$.
  Since $\zeta(z, 2) + \zeta(1 - z, 2) \in \mathcal{C}$, the algebraic dependence between $\zeta(z, 2), \zeta(1 - z, 2)$, and $z$ implies that $\zeta(z, 2)$ is algebraic over $\mathcal{C}(z)$.
  Let $P(z, t) \in \mathcal{C}(z)[t]$ be its minimal monic polynomial and we set $d := \deg_t P$. 
  By applying $\sigma$ to $P(z, \zeta(z, 2))$ and using the minimality of $P$, we show that
  \[
    P(z, t) = P(z + 1, t - 1 / z^2).
  \]
  Then the coefficient $g(z) \in \mathcal{C}(z)$ of $t^{d - 1}$ in $P$ satisfies $g(z + 1) = g(z) + \frac{d}{z^2}$.
  One can check (using, for example, the function ratpolysols in {\sc Maple}) that there is no such rational function for $d \neq 0$.
  Thus, $\trdeg_F E = 3$.
  
  Consider arbitrary $\alpha, \beta \in F$ and set $a := \alpha\zeta(z, 2) + \beta \zeta(1 - z, 2)$.
  \eqref{eq:zeta_diff} implies that $\sigma(a) \in \mathbb{C}(z)$, so $F\langle a\rangle \subset F(a, z)$.
  Hence $\trdeg_F F\langle a\rangle \leqslant 2$.
  Hence, $F\langle a \rangle \neq E$.
  Thus, such $a$ cannot be a primitive element of the extension $F \subset E$.
\end{example}

\begin{example}\label{ex:cannot_drop}
  This example shows that neither of the conditions~\ref{item:three_delta} and~\ref{item:three_sigma} in Theorem~\ref{thm:main} can be removed.
  We fix $s, t \in \mathbb{Z}_{\geqslant 0}$, $\Delta = \{\delta_1, \ldots, \delta_s\}$, and $\Sigma = \{\sigma_1, \ldots, \sigma_t\}$.
  Consider a free $\DS$-extension of $\mathbb{Q}$ with two generators, $x_1$ and $x_2$:
  \[
  \mathbb{Q} \subset E := \mathbb{Q} (\delta^{\bm{\alpha}}\sigma^{\bm{\beta}}x_i \mid \bm{\alpha} \in \mathbb{Z}_{\geqslant 0},\, \bm{\beta} \in \mathbb{Z}, \, i = 1, 2),
  \]
  where all $\delta^{\bm{\alpha}}\sigma^{\bm{\beta}}x_i$ are algebraically independent and $\Delta \cup \Sigma$ acts naturally (see~\eqref{eq:action_on_free}).
  \cite[Theorem~3.5.38]{KLMP} implies that this extension $\mathbb{Q} \subset E$ cannot be generated by one element as a $\DS$-field extension. 
  
  Let $\Delta_1 := \Delta \cup \{\delta_{s + 1}\}$.
  We choose rational numbers $c_1, \ldots, c_s$ and make $E$ a $\Delta_1{\hbox{-}}\Sigma$-field by defining $\delta_{s + 1}(a) := c_1\delta_1(a) + \ldots + c_s\delta_s(a)$ for every $a \in E$.
  For every $a \in E$, the subfield generated by $a$ using $\Delta_1 \cup \Sigma$ is the same as the subfield generated by $\Delta \cup \Sigma$.
  Thus, $\mathbb{Q} \subset E$ cannot be generated by one element as an extension of $\Delta_1{\hbox{-}}\Sigma$-fields.
  On the other hand, every $a \in E$ is $\Delta_1{\hbox{-}}\Sigma$-algebraic over $\mathbb{Q}$ because it satisfies $c_1\delta_1a + \ldots + c_s\delta_s a - \delta_{s + 1} a = 0$.
  Thus, the condition~\ref{item:three_delta} in Theorem~\ref{thm:main} cannot be removed.
  A similar argument with a superfluous automorphism $\sigma_{t + 1} := \sigma_1^{k_1} \sigma_2^{k_2}\ldots \sigma_t^{k_t}$, where $k_1, \ldots, k_t \in \mathbb{Z}$, show that the condition~\ref{item:three_sigma} cannot be removed.
\end{example}

%%%%%%%%%%%%%%%%%%%%%%%%

\section{Definitions and notation used in proofs}
\label{sec:def_proofs}

In the proofs, we will write $\delta^{\bm{\alpha}} \sigma^{\bm{\beta}} a$ instead of $\delta^{\bm{\alpha}} \sigma^{\bm{\beta}} (a)$

\begin{definition}[$\DS$-algebraicity]
  Let $F \subset E$ be an extension of $\DS$-fields. 
  An element $a \in E$ is \emph{$\DS$-algebraic} over $F$ if the set
    \[
    \{\delta^{\bm{\alpha}} \sigma^{\bm{\beta}} a \;|\; \bm{\alpha} \in \mathbb{Z}_{\geqslant 0}^s, \; \bm{\beta} \in \mathbb{Z}_{\geqslant 0}^t\}
    \]
    is algebraically dependent over $F$.
\end{definition}

\begin{definition}[Nondegenerate $\DS$-field]\label{def:nondegenerate}
  A $\DS$-field $E$ is called \emph{nondegenerate} if it satisfies conditions~\ref{item:three_delta} and~\ref{item:three_sigma} of Theorem~\ref{thm:main}, namely,
  \begin{itemize}
      \item[(3$\delta$)] $\delta_1, \ldots, \delta_s$ are linearly independent over $E$;
      \item[(3$\sigma$)] $\sigma_1, \ldots, \sigma_t$ are multiplicatively independent over $\mathbb{Z}$, that is, for every $\bm{\beta}\in \mathbb{Z}^t$, $\sigma^{\bm{\beta}}|_E = \operatorname{id} \iff \bm{\beta} = \bm{0}$.
  \end{itemize}
\end{definition}

\begin{definition}[$\DS$-constants]\label{def:constants}
  An element $a \in E$ of a $\DS$-field $E$ is said to be a \emph{constant} if $\delta_i a = 0$ for every $1 \leqslant i \leqslant s$ and $\sigma_j a = a$ for every $1\leqslant j \leqslant t$.
  Constants form a subfield in $E$.
  We will denote this subfield by $C(E)$.
\end{definition}

\begin{definition}[$\DS$-polynomials]
  Let $R$ be a $\DS$-ring.
  Consider the following ring of polynomials over $R$
  \[
    R\{x\} := R[ \delta^{\bm{\alpha}} \sigma^{\bm{\beta}} x \;|\; \bm{\alpha} \in \mathbb{Z}_{\geqslant 0}^s, \; \bm{\beta} \in \mathbb{Z}_{\geqslant 0}^t],
  \]
  where each $\delta^{\bm{\alpha}} \sigma^{\bm{\beta}} x$ is a separate variable.
  We can extend the structure of $\DS$-ring from $R$ to $R\{x\}$ by 
  \begin{equation}\label{eq:action_on_free}
    \delta_i \left( \delta^{\bm{\alpha}} \sigma^{\bm{\beta}} x \right) := \delta^{\bm{\alpha} + \bm{1}_i} \sigma^{\bm{\beta}}x \quad\text{ and }\quad \sigma_j \left( \delta^{\bm{\alpha}} \sigma^{\bm{\beta}} x \right) := \delta^{\bm{\alpha}} \sigma^{\bm{\beta} + \bm{1}_j} x,
  \end{equation}
  where $\bm{1}_i$ and $\bm{1}_j$ denote the $i$-th basis vector in $\mathbb{Z}^s$ and the $j$-th basis vector in $\mathbb{Z}^t$, respectively.
  Elements of $R\{x\}$ are called \emph{$\DS$-polynomials in $x$}.
\end{definition}

\begin{notation}
  Let $n$ be a positive integer.
  \begin{itemize}
      \item For $\bm{v} = (v_1, \ldots, v_n) \in \mathbb{Z}^{n}$, we define $|\bm{v}| := |v_1| + \ldots + |v_n|$.
      If $\bm{v} \in \mathbb{Z}_{\geqslant 0}^n$, then we also define $\bm{v}! := v_1! \ldots v_n!$.
      \item For a positive integer $m$, we define $\mathbb{Z}_{\geqslant 0}^n (m) := \{ \bm{v} \in \mathbb{Z}_{\geqslant 0}^n \;|\; |\bm{v}| \leqslant m  \}$.
  \end{itemize}
\end{notation}

\begin{definition}[Nonperiodic elements]\label{def:nonperiodic}
  Let $E$ be a $\DS$-field.
  \begin{itemize}
      \item An element $a \in E$ is said to be \emph{nonperiodic} if, for every nonzero $\bm{\beta} \in \mathbb{Z}^t$, $\sigma^{\bm{\beta}} a \neq a$.
      \item For a positive integer $m$, an element $a \in E$ is called \emph{$m$-nonperiodic} if 
  \[
  \left( \sigma^{\bm{\alpha}} a = \sigma^{\bm{\beta}} a \;\;\&\;\; \bm{\alpha}, \bm{\beta} \in \mathbb{Z}^t_{\geqslant 0}(m) \right) \implies \bm{\beta} = \bm{\alpha}.
  \]
  \end{itemize}
\end{definition}

\begin{notation}
  Let $K$ be a field and $K[\![x_1, \ldots, x_n]\!]$ be the formal power series ring over $K$.
  For an element 
  \[
  f = \sum\limits_{\bm{k} \in \mathbb{Z}^n_{\geqslant 0}} c_{\bm{k}} x^{\bm{k}} \in K[\![x_1, \ldots, x_n]\!],
  \]
  we denote its truncation at order $m$ by
  \[
  [f]_m := \sum\limits_{\bm{k} \in \{0, 1, \ldots, m\}^n} c_{\bm{k}} x^{\bm{k}} \in K[x_1, \ldots, x_n].
  \]
\end{notation}

%%%%%%%%%%%%%%%%%%%%%%%%

\section{Proof for differential fields}
\label{sec:differential_proof}

In this section, we consider the case $\Delta = \{\delta\}$ and $\Sigma = \varnothing$. 
We will denote $\delta a$ by $a'$ and say ``differential field'' and ``differentially algebraic'' instead of ``$\DS$-field'' and ``$\DS$-algebraic'', respectively. 

\begin{lemma}[Special case of Lemma~\ref{lem:operator}]\label{lem:operator_1d}
  Let $\partial = \frac{\partial}{\partial x}$ be the standard derivation on $K[\![x]\!]$.
  Let $D \in K[\partial]$ be a differential operator of order $m$.
  Then, for every $f \in K[\![ t ]\!]$,
  \[
  \bigl( Df = 0 \;\; \& \;\; [f]_m = 0 \bigr) \implies f = 0.
  \]
\end{lemma}

\begin{proof}
  Since $D$ has $\partial$-constant coefficients, every solution of $D$ is uniquely defined by its first $m$ Taylor coefficients.
  Since $f$ has the same first $m$ Taylor coefficients as the zero solution, $f = 0$.
\end{proof}

For the rest of the section, for a differential field $E$, we extend the derivation from $E$ to $E[\![ x ]\!]$ by 
\begin{equation}\label{eq:extend_der_diff}
\bigl( \sum\limits_{i = 0}^\infty c_i x^i \bigr)' = \sum\limits_{i = 0}^\infty c_i'x^i \quad \text{ for every } \sum\limits_{i = 0}^\infty c_i x^i \in E[\![x]\!].
\end{equation}

\begin{lemma}[Special case of Lemma~\ref{lem:independence}]\label{lem:independence_ordinary}
  Let $E$ be a  differential field.
  Let $a \in E$ be a nonconstant element.
  For $m \in \mathbb{Z}_{\geqslant 0}$, we introduce the following subset of $E[\![x]\!]$ (with the derivation defined in~\eqref{eq:extend_der_diff})
  \[
  S_{m} := \left\{ (e^{a x})^{(r)} \;|\; 0 \leqslant r \leqslant m \right\}.
  \]
  Then the elements of $S_m$ are linearly independent over $E$.
\end{lemma}

\begin{proof}
  We will prove the lemma by induction on $m$.
  The base case $m = 0$ is true because $S_0 = \{e^{ax}\}$.
  Assume that we have proved the lemma for some $m \geqslant 0$.
  Let $V$ be a space of all polynomials from $E[x]$ of degree at most $m$.
  Then $S_m \subset V_e := Ve^{ax}$.
  Hence it is sufficient to prove that $S_{m + 1}\setminus S_m = \{(e^{ax})^{(m + 1)}\}$ does not belong to $V_e$.
  This is true because
  \[
    (e^{ax})^{(m + 1)} \equiv (a'x)^{m + 1}e^{ax} \pmod{V_e}\quad \text{ and }\quad a' \neq 0.\qedhere
  \]
\end{proof}

\begin{lemma}[Special case of Lemma~\ref{lem:core}]\label{lem:core_differential}
  Let $F \subset E$ be an extension of differential fields, and the derivation on $E[\![x]\!]$ is defined as in~\eqref{eq:extend_der_diff}.
  Let $a \in E$ be a nonconstant element such that there exists a nontrivial $F$-linear combination of the truncations
  \[    
    [e^{ax}]_{2(m + 1)},\;\; [(e^{ax})']_{2(m + 1)},\;\; \ldots,\;\; [(e^{ax})^{(m)} ]_{2(m + 1)}
  \]
  that belongs to $F[x]$. Then $a \in F$.
\end{lemma}

\begin{proof}
  We are given that there exist $c_0, \ldots, c_m \in F$ not all zero such that 
  \[
  c_0 [e^{ax}]_{2(m + 1)} + \ldots + c_m [(e^{ax})^{(m)}]_{2(m + 1)} = f \in F[x].
  \]
 
  \begin{step}\label{step:1_differential}
  There exists a nonzero polynomial $C(x) \in E[x]$ of degree at most $m$ such that 
  \[
    c_0 e^{ax} + \ldots + c_m (e^{ax})^{(m)} = C(x) e^{ax}.
  \]
  \end{step}
  
  \noindent
  Note that every $E$-linear combination of $e^{ax}, \ldots, (e^{ax})^{(m)}$ is a product of $e^{ax}$ and an element of $E[x]$ of degree at most $m$.
  Since not all $c_0, \ldots, c_m$ are zeros, Lemma~\ref{lem:independence_ordinary} implies that $C(x) \neq 0$.

  \begin{step}\label{step:2_differential}
    Let $\overline{E} \supset E$ be an algebraic closure of $E$.
    For every field automorphism $\tau\colon \overline{E} \to \overline{E}$ such that $\tau|_F = \operatorname{id}$, we have $\tau(a) = a$.
  \end{step}
  
  \noindent
  Let $S := C e^{a x} - \tau(C) e^{\tau(a)x}$.
  Since $\tau|_F = \operatorname{id}$, $[S]_{2(m + 1)} = 0$.
  Since $\deg C(x) \leqslant m$, we have
  \[
  D S = 0, \quad \text{ where } \quad D := (\partial - a)^{m + 1}(\partial - \tau(a))^{m + 1}.
  \]
  Since the order of $D$ is $2(m + 1)$ and $[S]_{2(m + 1)} = 0$, Lemma~\ref{lem:operator_1d} implies that $S = 0$.
  Then $e^{a x}$ and $e^{\tau(a) x}$ are linearly dependent over $E(x)$.
  This is possible only if $a = \tau(a)$.
  
  \begin{step}\label{step:3_differential}
    $a \in F$.
  \end{step}
  
  \noindent
  If $a \not\in F$, then there exists an automorphism $\tau\colon \overline{E} \to \overline{E}$ such that $\tau|_F = \operatorname{id}$ and $a \neq \tau(a)$.
  This is impossible due to Step~\ref{step:2_differential}.
\end{proof}
  
  \begin{theorem}
  Let $F \subset E$ be an extension of differential fields such that 
  \begin{enumerate}[label=(\arabic*)]
      \item there exist $a_1, \ldots, a_n \in E$ such that $E = F\langle a_1, \ldots, a_n \rangle$;
      \item for every $1 \leqslant j \leqslant n$, $a_j$ is differentially algebraic over $F$;
      \item $E$ contains a nonconstant element.
  \end{enumerate}
  Then there exists $a \in E$ such that $E = F\langle a \rangle$.
  \end{theorem}
  
  \begin{proof}
    Since each of $a_1, \ldots, a_n$ is differentially  algebraic over $F$, $M := \trdeg_F E < \infty$~\cite[Corollary~1, p.~112]{KolchinBook}.
    We will prove by induction on $\ell$ that, for every $0 \leqslant \ell \leqslant n$, there exists $b_\ell \in E$ such that
    \begin{itemize}
      \item $b_\ell' \neq 0$;
      \item $E = F\langle b_\ell, a_{\ell + 1}, \ldots, a_n \rangle$.
    \end{itemize}
    Since $E = F\langle b_n\rangle$, proving the existence of such $b_0, \ldots, b_n$ will prove the theorem.
    
    For the base case $\ell = 0$, we choose $b_0$ to be any nonconstant element of $E$.
    Assume that we have constructed $b_{\ell}$ for some $\ell \geqslant 0$.
    We introduce a set of variables
    \[
      \Theta := \{ \theta_i \;|\; -1 \leqslant i \leqslant 2(M + 1)\}
    \]
    and extend the derivation from $E$ to $E[\Theta]$ by making the elements of $\Theta$ constants.
    Let
    \begin{equation}\label{eq:primitive_differential}
    B_{\ell + 1} := \theta_{-1} a_{\ell + 1} + \sum\limits_{i = 0}^{2(M + 1)} \frac{\theta_i}{i!} b_{\ell}^i.
    \end{equation}
     We regard any point $\varphi \in \mathbb{Q}^{|\Theta|}$ as a function $\varphi\colon \Theta \to \mathbb{Q}$ and extend it to a $E$-algebra homomorphism $\varphi\colon E[\Theta] \to E$.
  
  \begin{claim}
    There exists a Zariski open nonempty subset $U_1  \subset \mathbb{Q}^{|\Theta|}$ such that 
  \[
  F\langle \varphi(B_{\ell + 1}) \rangle = F\langle b_{\ell}, a_{\ell + 1} \rangle \quad \text{for every } \varphi \in U_1.
  \]
  \end{claim}
  \noindent
  Since
  \[
  \trdeg_{F(\Theta)} F(\Theta, B_{\ell + 1}, B_{\ell + 1}', \ldots, B_{\ell + 1}^{(M)} ) \leqslant \trdeg_F F\langle a_{\ell + 1}, b_{\ell} \rangle \leqslant M,
  \]
  $B_{\ell + 1}, B_{\ell + 1}', \ldots, B_{\ell + 1}^{(M)}$ are algebraically dependent over $F(\Theta)$.
  Thus, there exists a differential polynomial $R \in F(\Theta)[ z, z', \ldots, z^{(M)} ]$ such that $R(B_{\ell + 1}) = 0$.
  We will assume that $R$ is chosen to be of the minimal possible total degree.
  We introduce
  \[
  R_{i} := \frac{\partial R}{\partial z^{(i)}} (B_{\ell + 1}) \text{ for } 0 \leqslant i \leqslant M \quad\text{ and }\quad R_{\theta_{j}} := \frac{\partial R}{\partial \theta_{j}} (B_{\ell + 1}) \text{ for } -1 \leqslant j \leqslant 2(M + 1).
  \]
  The minimality of the degree of $R$ implies that not all of $R_{i}$ are zero.
  Consider any $0 \leqslant j \leqslant 2(M + 1)$.
  Differentiating $R(B_\ell) = 0$ with respect to $\theta_{j}$, we obtain
  \begin{equation}\label{eq:derivative_differential}
  \sum\limits_{i = 0}^M R_{i} \frac{(b_{\ell}^j)^{(i)}}{j!} = -R_{\theta_{j}}.
  \end{equation}
  Consider the power series ring $E[\![x]\!]$ with the derivation defined in~\eqref{eq:extend_der_diff}.
  We multiply~\eqref{eq:derivative_differential} by $x^j$ and sum such equations over all $0 \leqslant j \leqslant 2(M + 1)$.
  We obtain
  \begin{equation}\label{eq:with_exponents_differential}
    \sum\limits_{i = 0}^M R_{i} [(e^{b_{\ell} x})^{(i)}]_{2(M + 1)} = \sum\limits_{j = 0}^{2(M + 1)} -R_{\theta_{j}}x^j.
  \end{equation}
  We apply Lemma~\ref{lem:core_differential} to~\eqref{eq:with_exponents_differential} with $a = b_\ell$ and $F = F\langle \Theta, B_{\ell + 1} \rangle$, and deduce that $b_\ell \in F\langle \Theta, B_{\ell + 1} \rangle$.
  Then there exist nonzero differential polynomials $P_1, P_2 \in F[\Theta]\{z\}$ such that 
  \begin{equation}\label{eq:formula_b_ell_differential}
  b_\ell = \frac{P_1(B_\ell)}{P_2(B_{\ell})}.
  \end{equation}
  We define $U_1 := \{\varphi \in \mathbb{Q}^{|\Theta|} \;|\; \varphi(P_2(B_{\ell})) \neq 0 \text{ and } \varphi(\theta_{-1}) \neq 0\}$.
  Since $P_2$ is a nonzero polynomial, $U_1$ is nonempty.
  For every $\varphi \in U_1$, \eqref{eq:formula_b_ell} implies that $b_\ell \in F\langle \varphi(B_{\ell+ 1}) \rangle$.
  Since $\varphi(\theta_{-1}) \neq 0$, $a_{\ell + 1} \in F\langle \varphi(B_{\ell+ 1}) \rangle$, so $F\langle \varphi(B_{\ell + 1}) \rangle = F\langle b_{\ell}, a_{\ell + 1} \rangle$.
  The claim is proved.
  
  \begin{claim}
    Let $U_2 := \{ \varphi \in \mathbb{Q}^{|\Theta|} \;|\; \varphi(B_{\ell + 1})' \neq 0\}$.
    Then $U_2$ is a nonempty Zariski open set.
  \end{claim}
  \noindent
  Since $U_2$ is defined by an inequation, it is open.
  Consider $\varphi_0 \in \mathbb{Q}^{|\Theta|}$ defined by $\varphi_0(\theta_1) = 1$ and $\varphi(\theta_j) = 0$ for $j \neq 1$.
  Then $\varphi_0(B_{\ell + 1}) = b_\ell$.
  Thus, $\varphi_0 \in U_2$, $U_2 \neq \varnothing$.
  The claim is proved.
  
  We finish the proof by considering $\varphi \in U_1 \cap U_2$ and defining $b_{\ell + 1} := \varphi(B_{\ell + 1})$.
  \end{proof}  

%%%%%%%%%%%%%%%%%%%%%%%%%%%%%%%%%%%%

\section{Proof for the general case}
\label{sec:proof_main}

\subsection{Choosing a sufficiently nonconstant element}

\begin{notation}
  Let $F$ be $\DS$-field.
  For $a_1, \ldots, a_n \in F$, we denote their Jacobian matrix by
  \[
  J(a_1, \ldots, a_n) := 
  \begin{pmatrix}
    \delta_1 a_1 & \delta_2 a_1 & \ldots & \delta_s a_1 \\
    \delta_1 a_2 & \delta_2 a_2 & \ldots & \delta_s a_2 \\
    \vdots & \vdots & \ddots & \vdots \\
    \delta_1 a_n & \delta_2 a_s & \ldots & \delta_s a_n 
  \end{pmatrix}.
  \]
  For $n = s = 0$, we will use a convention $\det J(a_1, \ldots, a_n) = 1$.
\end{notation}

%%%%%%%%%%%%%%%%%%

\begin{lemma}\label{lem:nondegenerate_field_has_nonzero_Jac}
  Let $E$ be a $\DS$-field.
  Then the following statements are equivalent
  \begin{enumerate}[label=(\arabic*), itemsep=0pt, topsep=1pt]
      \item\label{cond:nonzeroder} $\delta_1, \ldots, \delta_s$ are linearly independent over $C(E)$ (see Definition~\ref{def:constants});
      \item $\delta_1, \ldots, \delta_s$ are linearly independent over $E$;
      \item\label{cond:nonzeroJac} there exist $a_1, a_2, \ldots, a_s \in E$ such that $\det J(a_{1}, \ldots, a_{s}) \neq 0$.
  \end{enumerate}
\end{lemma}

\begin{proof}
  \ref{cond:nonzeroJac} $\implies$ \ref{cond:nonzeroder}. Assume that \ref{cond:nonzeroder} does not hold.
  Then there exist $\mathbf{b} = (b_1, \ldots, b_s) \in E^s$ such that $\delta|_E = 0$, where $\delta := b_1\delta_1 + \ldots + b_s\delta_s$.
  We have
  \[
  \delta (a_{1}, \ldots, a_{s})^{T} = J(a_{1}, \ldots, a_{s}) \mathbf{b}^{T}.
  \]
  Since $J(a_{1}, \ldots, a_{s})$ is nondegenerate, the latter is nonzero for every nonzero $\mathbf{b}$.
  Thus, $\delta a_{j}$ is nonzero for at least one $1 \leqslant j \leqslant s$, so we arrived at the contradiction.
  
  \ref{cond:nonzeroder} $\implies$ \ref{cond:nonzeroJac}. 
  Let $r$ be the maximal integer such that there exist $a_1, \ldots, a_r \in E$ such that $J(a_1, \ldots, a_r)$ has rank $r$.
  If $r = s$, then we are done.
  If $r < s$ then we will arrive at the contradiction with~\ref{cond:nonzeroder} in the two following steps.
  
  \begin{enumerate}[leftmargin=0cm,labelsep=0cm,align=left,label=\textbf{Step {\arabic*}}:\ ]
  \item \emph{There exist $b_1, \ldots, b_s \in E$ not all zero such that $b_1\delta_1 + \ldots + b_s\delta_s$ defines a zero derivation on $E$.}
  Reenumerating $\delta_1, \ldots, \delta_s$ if necessary, we can assume that the first $r$ columns in $J(a_1, \ldots, a_r)$ are linearly independent.
  For every $1 \leqslant i \leqslant r + 1$, we denote the determinant of the matrix consisting of the first $r + 1$ columns of $J(a_1, \ldots, a_r)$ except the $i$-th by $A_i$.
  Then $A_{r + 1} \neq 0$.
  Consider an arbitrary $a \in E$. 
  The maximality of $r$ implies that $\rank J(a, a_1, \ldots, a_r) = r$, so every $(r + 1) \times (r + 1)$-minor of $J(a, a_1, \ldots, a_r)$ is degenerate.
  Expanding the determinant of the matrix consisting of the first $r + 1$ columns of $J(a, a_1, \ldots, a_r)$ along the first row, we obtain
  \[
    A_1(\delta_1a) - A_2 (\delta_2 a) + \ldots  + (-1)^{r}A_{r + 1}(\delta_r a) = 0.
  \]
  Since $A_{r + 1} \neq 0$, $A_1\delta_1 - A_2\delta_2 + \ldots + (-1)^rA_{r + 1} \delta_r$ is a nontrivial $E$-linear combination of $\delta_1, \ldots, \delta_s$ that defines a zero derivation on $E$.
  
  \item \emph{There exist $b_1, \ldots, b_s \in C(E)$ not all zero such that $b_1\delta_1 + \ldots + b_s\delta_s$ is a zero derivation on~$E$.}
  Among all nontrivial linear combinations of $\delta_1, \ldots, \delta_s$ defining a zero derivation on $E$, consider a combination with the minimal possible number, say $q$, of nonzero coefficients.
  Reenumerating $\delta_1, \ldots, \delta_s$, we can assume that this combination is of the form $\delta = b_1\delta_1 + \ldots + b_q \delta_q$ for some nonzero $b_1, \ldots, b_q \in E$.
  Moreover, by dividing the combination by $b_1$, we can further assume that $b_1 = 1$.
  If $b_2, \ldots, b_q \in C(E)$, then we are done.
  If at least one of them, say $b_2$, does not belong to $C(E)$, then there are two options:
  \begin{itemize}
      \item There exists $1 \leqslant i \leqslant s$ such that $\delta_i b_2 \neq 0$.
      Then consider
      \[
      [\delta_i, \delta] = (\delta_i b_2) \delta_2 + \ldots + (\delta_i b_q) \delta_q.
      \]
      Then $[\delta_i, \delta]$ is a nontrivial $E$-linear combination of $\delta_1, \ldots, \delta_s$ such that $[\delta_i, \delta]|_E = 0$.
      This contradicts the minimality of $q$.
      
      \item There exists $1 \leqslant i \leqslant t$ such that $\sigma_i b_2 \neq b_2$.
      Then consider
      \[
      \sigma_i\delta\sigma_i^{-1} - \delta = (\sigma_i b_2 - b_2) \delta_2 + \ldots + (\sigma_i b_q - b_q) \delta_q.
      \]
      Then $\sigma_i\delta\sigma_i^{-1} - \delta$ is a nontrivial $E$-linear combination of $\delta_1, \ldots, \delta_s$ such that $(\sigma_i\delta\sigma_i^{-1} - \delta)|_E = 0$.
      This contradicts the minimality of $q$.
  \end{itemize}
  \end{enumerate}
\end{proof}

%%%%%%%%%%%%%%%

\begin{lemma}\label{lem:nonzeroJac_one_element}
  Let $E$ be a $\DS$-field. Let $a_1, \ldots, a_n$ be elements of $E$ such that $s \leqslant n$ and $\det J(a_1, \ldots, a_s) \neq 0$.
  Then there exists a nonempty Zariski open subset $U \subset \{ P \in \mathbb{Q}[x_1, \ldots, x_n] \;|\; \deg P \leqslant 2\}$ such that, for every $P \in U$,
  \begin{equation}\label{eq:nonzeroJac}
  \det J(\delta_1 P_a, \ldots, \delta_s P_a) \neq 0, \text{ where } P_a := P(a_1, \ldots, a_n).
  \end{equation}
\end{lemma}

\begin{proof}
  The inequation~\eqref{eq:nonzeroJac} defines an open subset in $\{ P \in \mathbb{Q}[x_1, \ldots, x_n] \;|\; \deg P \leqslant 2\}$.
  It remains to show that this subset is nonempty.
  We introduce new variables $\Lambda := \{\lambda_1, \ldots, \lambda_s\}$ and set $\delta_i \lambda_j = 0$ for every $1 \leqslant i, j \leqslant s$. 
  Consider
  \[
  P(x_1, \ldots, x_n) = \sum\limits_{\ell = 1}^s \lambda_\ell a_\ell + a_1^2 + \ldots + a_s^2.
  \]
  Then $J(\Lambda) := \det J(\delta_1 P_a, \ldots, \delta_s P_a) \in E[\Lambda]$.
  We will consider $J(\Lambda)$ as a polynomial in $\Lambda$ over $E$ and show that $J(\Lambda) \neq 0$.
  $J(\Lambda)$ is the determinant of the matrix whose $(i, j)$-th entry is 
  \begin{equation}\label{eq:ij_entry}
  \delta_i\delta_j \left( \sum\limits_{\ell = 1}^s \lambda_\ell a_\ell + a_1^2 + \ldots + a_s^2 \right) = \sum\limits_{\ell = 1}^s \lambda_\ell \delta_i\delta_j a_\ell + \sum\limits_{\ell = 1}^s (2a_\ell\delta_i\delta_j a_\ell + 2(\delta_i a_\ell) (\delta_j a_\ell)).
  \end{equation}
  If we set $\lambda_\ell = -2a_\ell$ for $1 \leqslant \ell \leqslant s$, then the right-hand side of~\eqref{eq:ij_entry} can be written as
  \[
  \sum\limits_{\ell = 1}^s -2a_\ell \delta_i\delta_j a_\ell + \sum\limits_{\ell = 1}^s (2a_\ell\delta_i\delta_j a_\ell + 2(\delta_i a_\ell) (\delta_j a_\ell)) = 2 \sum\limits_{\ell = 1}^s (\delta_i a_\ell) (\delta_j a_\ell).
  \]
  Thus, we can write
  \[
  J(-2a_1, \ldots, -2a_s) = 2 \det\left( J(a_1, \ldots, a_s) J^T(a_1, \ldots, a_s) \right) \neq 0.
  \]
  Since $J(\Lambda)$ is a nonzero polynomial, there exist $\lambda_1^\ast, \ldots, \lambda_s^\ast \in \mathbb{Q}$ such that $J(\lambda_1^\ast, \ldots, \lambda_s^\ast) \neq 0$.
  Then $P^\ast := \lambda_1^\ast x_1 + \ldots + \lambda_s^\ast x_s + x_1^2 + \ldots + x_s^2$ is a witness of the nonemptyness of $U$.
\end{proof}

%%%%%%%%%%%%%%%%%%%%%%%

\begin{lemma}\label{lem:existence_nonperiodic_and_nonzeroJac}
  Consider an extension of $\DS$ fields $F \subset E$ such that 
  \begin{itemize}
    \item $E = F\langle a_1, \ldots, a_n \rangle$ with $s \leqslant n$;
    \item $E$ is nondegenerate (see Definition~\ref{def:nondegenerate});
    \item $\det J(a_1, \ldots, a_s) \neq 0$.
  \end{itemize}
  Then, for every $m$, there exists a polynomial $P \in F[x_1, \ldots, x_n]$ of degree at most two
  such that $P_a := P(a_1, \ldots, a_n)$ is $m$-nonperiodic (see Definition~\ref{def:nonperiodic}) and $\det J(\delta_1 P_a, \ldots, \delta_s P_a) \neq 0$.
\end{lemma}

\begin{proof}
  We will extend the set $a_1, \ldots, a_n$ of generators of $E$ over $F$ by some elements of $F$ as follows. 
  For every pair $\bm{\alpha}, \bm{\beta} \in \mathbb{Z}_{\geqslant 0}^t (m)$ such that $\bm{\alpha} \neq \bm{\beta}$, since $\sigma^{\bm{\alpha} - \bm{\beta}}|_E \neq \operatorname{id}$, there are two options:
  \begin{itemize}
      \item if $\sigma^{\bm{\alpha} - \bm{\beta}}|_F \neq \operatorname{id}$, then we take $a \in F$ such that $\sigma^{\bm{\alpha}} a \neq \sigma^{\bm{\beta}} a$ and add it to the set of generators;
      \item otherwise, if $\sigma^{\bm{\alpha} - \bm{\beta}}|_F = \operatorname{id}$, there exists $1 \leqslant i \leqslant n$ such that $\sigma^{\bm{\alpha} - \bm{\beta}} a_i \neq a_i$.
  \end{itemize}
  Using this procedure we construct an extended set of generators $a_1, \ldots, a_N$ such that 
  \begin{itemize}
      \item $a_{n + 1}, \ldots, a_N \in F$ and
      \item for every pair $\bm{\alpha}, \bm{\beta} \in \mathbb{Z}_{\geqslant 0}^t (m)$ such that $\bm{\alpha} \neq \bm{\beta}$, there exists $1 \leqslant i(\bm{\alpha}, \bm{\beta}) \leqslant N$ such that $\sigma^{\bm{\alpha}} a_{ i(\bm{\alpha}, \bm{\beta}) } \neq \sigma^{\bm{\beta}} a_{ i(\bm{\alpha}, \bm{\beta}) }$.
  \end{itemize}
  Let $V := \{P \in \mathbb{Q}[x_1, \ldots, x_N] \;|\; \deg P \leqslant 2\}$.
  Let $U \subset V$ be a nonempty open Zariski subset given by Lemma~\ref{lem:nonzeroJac_one_element}.
  
  For every pair $\bm{\alpha}, \bm{\beta} \in \mathbb{Z}_{\geqslant 0}^t (m)$ such that $\bm{\alpha} \neq \bm{\beta}$, consider
  \[
  U_{\bm{\alpha}, \bm{\beta}} := \{ P \in V \;|\; \sigma^{\bm{\alpha}}P(a_1, \ldots, a_N) \neq \sigma^{\bm{\beta}}P(a_1, \ldots, a_N)\} \subset V.
  \]
  Since $U_{\bm{\alpha}, \bm{\beta}}$ is defined by an inequation, it is an open subset of $V$.
  Moreover, since $x_{ i(\bm{\alpha}, \bm{\beta}) } \in U_{\bm{\alpha}, \bm{\beta}}$, $U_{\bm{\alpha}, \bm{\beta}} \neq \varnothing$.
  Then the intersection of $U$ and all the subsets $U_{\bm{\alpha}, \bm{\beta}}$ with $\bm{\alpha}, \bm{\beta} \in \mathbb{Z}_{\geqslant 0}^t$ and $\bm{\alpha} \neq \bm{\beta}$
  is a nonempty open subset of $V$.
  Let $P_0$ be an element of this subset.
  Then $P_1 := P_0(x_1, \ldots, x_n, a_{n + 1}, \ldots, a_N) \in F[x_1, \ldots, x_n]$ is a desired polynomial.
\end{proof}

%%%%%%%%%%%%%%%%%%%%%%%%%%%%%%%%%%%%

\subsection{Core lemmas}

The following lemma generalizes Lemma~\ref{lem:operator_1d}.

\begin{lemma}\label{lem:operator}
  Let $K$ be a field. 
  We denote the partial derivatives of $K[\![ x_1, \ldots, x_n ]\!]$ with respect to $x_1, \ldots, x_n$ by $\partial_1, \ldots, \partial_n$, respectively.
  Let $D_1 \in K[\partial_{1}], \ldots, D_n \in K[\partial_n]$ be differential operators of order at most $m$.
  For every $f \in K[\![ x_1, \ldots, x_n ]\!]$,
  \[
  \bigl( D_1f = \ldots = D_n f = 0 \;\; \& \;\; [f]_m = 0 \bigr) \implies f = 0.
  \]
\end{lemma}

\begin{proof}
  Let 
  \[
  f = \sum\limits_{\bm{k} \in \mathbb{Z}^n_{\geqslant 0}} c_{\bm{k}} x^{\bm{k}} \in K[\![x_1, \ldots, x_n]\!].
  \]
  We will show that $c_{\bm{k}} = 0$ for every $\bm{k} \in \mathbb{Z}_{\geqslant 0}^n$ by induction on $|\bm{k}|$.
  Let $\bm{k} = (k_1, \ldots, k_n)$. 
  If $k_i < m$ for every $1 \leqslant i \leqslant n$, then $c_{\bm{k}} = 0$ because $[f]_m = 0$.
  Assume that there exists $1 \leqslant i \leqslant n$ such that $k_i \geqslant m$.
  Then $D_i f = 0$ implies that $c_{\bm{k}}$ is a linear combination of $c_{\bm{k} - \bm{1}_i}, \ldots, c_{\bm{k} - m\bm{1}_i}$, where $\bm{1}_i$ is the $i$-th basis vector of $\mathbb{Z}^n$.
  Due to the induction hypothesis, these coefficients are all equal to zero, so $c_{\bm{k}} = 0$.
\end{proof}

%%%%%%%%%%%%%%%%%%%%%%%%%%%%%%

For every positive integer $n$, throughout the rest of the paper, for a $\DS$-field $E$, we extend the operators from $E$ to $E[\![ x_1, \ldots, x_{n} ]\!]$ by 
\begin{equation}\label{eq:extend_der_gen}
\phi\bigl( \sum\limits_{\bm{k} \in \mathbb{Z}^{n}_{\geqslant 0}}^\infty c_{\bm{k}} x^{\bm{k}} \bigr) = \sum\limits_{\bm{k} \in \mathbb{Z}^{n}_{\geqslant 0}}^\infty \phi(c_{\bm{k}})x^{\bm{k}} \quad \text{ for every } \sum\limits_{\bm{k} \in \mathbb{Z}^{n}_{\geqslant 0}}^\infty c_{\bm{k}} x^{\bm{k}} \in E[\![ x_1, \ldots, x_{n} ]\!] \text{ and } \phi \in \Delta \cup \Sigma.
\end{equation}

The following lemma generalizes Lemma~\ref{lem:independence_ordinary}.

{
\begin{lemma}\label{lem:independence}
  Let $E$ be a $\DS$-field
  and $a_1, \ldots, a_n$ be elements of $E$ such that $\rank J(a_1, \ldots, a_n) = s$.
  We extend $\Delta \cup\Sigma$ to $E[\![x_1, \ldots, x_n]\!]$ as in~\eqref{eq:extend_der_gen}.
  For $m \in \mathbb{Z}_{\geqslant 0}$, we introduce the following subset of $E[\![x_1, \ldots, x_n]\!]$
  \[
  S_{m} := \left\{ \delta^{\bm{\alpha}} e^{a_1x_1 + \ldots + a_nx_n} \;|\; \bm{\alpha} \in \mathbb{Z}_{\geqslant 0}^s(m) \right\}.
  \]
  Then the elements of $S_m$ are linearly independent over $E$.
\end{lemma}
}

{
\begin{proof}
  We will use notation $(\mathbf{a}, \mathbf{x}) := a_1x_1 + \ldots + a_nx_n$.
  We will prove the lemma by induction on $m$.
  The base case $m = 0$ is true because $S_0 = \{e^{(\mathbf{a}, \mathbf{x})}\}$.
  
  Assume that we have proved the lemma for some $m \geqslant 0$.
  Let $V$ be a space of all polynomials from $E[x_1, \ldots, x_n]$ of degree at most $m$.
  Then $S_m \subset V_e := Ve^{(\mathbf{a}, \mathbf{x})}$.
  Hence it is sufficient to prove that the elements of $S_{m + 1} \setminus S_m$ are linearly independent modulo $V_e$.
  For every $\bm{\alpha} = (\alpha_1, \ldots, \alpha_s) \in \mathbb{Z}_{\geqslant 0}^s(m + 1)$, we have
  \[
    \delta^{\bm{\alpha}}e^{(\mathbf{a}, \mathbf{x})} \equiv \prod\limits_{i = 1}^s (\delta_i (\mathbf{a}, \mathbf{x}))^{\alpha_i}e^{(\mathbf{a}, \mathbf{x})} \pmod{V_e}.
  \]
  Thus, it is sufficient to prove that the elements of $\{ \prod\limits_{i = 1}^s (\delta_i (\mathbf{a}, \mathbf{x}))^{\alpha_i} \mid \bm{\alpha} \in \mathbb{Z}_{\geqslant 0}^s(m + 1)\}$ are linearly independent over $E$.
  Assume the contrary.
  Then there exists a nonzero homogeneous polynomial $P \in E[y_1, \ldots, y_s]$ of degree $m + 1$ such that 
  \[
    P(\delta_1(\mathbf{a}, \mathbf{x}), \ldots, \delta_s (\mathbf{a}, \mathbf{x})) = 0.
  \]
  However, since $\rank J(a_1, \ldots, a_n) = s$, linear forms $\delta_1 (\mathbf{a}, \mathbf{x}), \ldots, \delta_s(\mathbf{a}, \mathbf{x})$ are linearly independent, so $P$ cannot vanish on them.
\end{proof}}

%%%%%%%%%%%%%%%%%%%%%%%%%%%%%%%%%%%%%

The following lemma generalizes Lemma~\ref{lem:core_differential}.

\begin{lemma}\label{lem:core}
  Let $F \subset E$ be an extension of $\DS$-fields.
  Let $a_1, \ldots, a_{n}$ be elements of $E$ such that
  \begin{itemize}
      \item ${\rank} J(a_1, \ldots, a_{n}) { = s}$;
      \item $a_1$ is $2M$-nonperiodic.
  \end{itemize}
  Let $N = 2(M + 1)^{t + 1}$. 
  We extend $\Delta \cup \Sigma$ to $E[\![x_1, \ldots, x_{n}]\!]$ as in~\eqref{eq:extend_der_gen}.
  If there exists a nontrivial $F$-linear combination of
  \begin{equation}\label{eq:set_functions}
  \left\{ [\delta^{\bm{\alpha}}\sigma^{\bm{\beta}} e^{a_1x_1 + \ldots + a_{n} x_{n}}]_{N} \;|\; \bm{\alpha} \in \mathbb{Z}_{\geqslant 0}^s(M),\; \bm{\beta} \in \mathbb{Z}_{\geqslant 0}^t(M) \right\}
  \end{equation}
  that belongs to $F[x_1, \ldots, x_{n}]$, then $a_1 \in F$.
\end{lemma}

\begin{proof}
  We are given that there exist $c_{\bm{\alpha}, \bm{\beta}} \in F$ not all zero such that
  \begin{equation}\label{eq:dependence}
    \sum\limits_{\substack{\bm{\alpha} \in \mathbb{Z}_{\geqslant 0}^s(M),\\  \bm{\beta} \in \mathbb{Z}_{\geqslant 0}^t(M)}} c_{\bm{\alpha}, \bm{\beta}} [\delta^{\bm{\alpha}}\sigma^{\bm{\beta}} e^{a_1x_1 + \ldots + a_{n} x_{n}}]_{N} = \left[ \sum\limits_{\substack{\bm{\alpha} \in \mathbb{Z}_{\geqslant 0}^s(M),\\  \bm{\beta} \in \mathbb{Z}_{\geqslant 0}^t(M)}} c_{\bm{\alpha}, \bm{\beta}} \delta^{\bm{\alpha}}\sigma^{\bm{\beta}} e^{a_1x_1 + \ldots + a_{n} x_{n}} \right]_{N} = f \in F[x_1, \ldots, x_{n}].
  \end{equation}
  Collecting together the terms with the same exponential part, we can write
  \[
  \sum\limits_{\substack{\bm{\alpha} \in \mathbb{Z}_{\geqslant 0}^s(M),\\  \bm{\beta} \in \mathbb{Z}_{\geqslant 0}^t(M)}} c_{\bm{\alpha}, \bm{\beta}} \delta^{\bm{\alpha}}\sigma^{\bm{\beta}} e^{a_1x_1 + \ldots + a_{n} x_{n}} = \sum\limits_{\bm{\beta} \in \mathbb{Z}_{\geqslant 0}^t(M)} C_{\bm{\beta}} e^{(\sigma^{\bm{\beta}}a_1) x_1 + \ldots + (\sigma^{\bm{\beta}}a_{n}) x_{n}},
  \]
  where $C_{\bm{\beta}} \in E[x_1, \ldots, x_{n}]$.
  
  \begin{step}
  There exists $\bm{\beta} \in \mathbb{Z}_{\geqslant 0}^t(M)$ such that $C_{\bm{\beta}} \neq 0$.
  \end{step}
  \noindent
  Consider any $\bm{\beta}$ such that $C_{\bm{\beta}} = 0$.
  Since
  \[
  {\rank} J(a_1, \ldots, a_{n}) {= s} \implies {\rank} J(\sigma^{\bm{\beta}}a_1, \ldots, \sigma^{\bm{\beta}}a_{n}) {= s},
  \]
  {Lemma~\ref{lem:independence}} implies that $\{ \delta^{\bm{\alpha}} e^{(\sigma^{\bm{\beta}} a_1)x_1 + \ldots + (\sigma^{\bm{\beta}} a_{n}) x_{n}} \;|\; \bm{\alpha} \in \mathbb{Z}_{\geqslant 0}^s(M) \}$ are $E$-linearly independent.
  Thus, $C_{\bm{\beta}} = 0$ implies that, for every $\bm{\alpha} \in \mathbb{Z}_{\geqslant 0}^s(M)$, $c_{\bm{\alpha}, \bm{\beta}} = 0$.
  Thus, if $C_{\bm{\beta}} = 0$ for every $\bm{\beta} \in \mathbb{Z}_{\geqslant 0}^t(M)$, we arrive at contradiction with the fact that not all $c_{\bm{\alpha}, \bm{\beta}}$ are zeros.
  
  \begin{step}\label{step:2_general}
  Let $\overline{E}$ be an algebraic closure of $E$ (we do not assume that $\overline{E}$ has a structure of $\DS$-field).
  Let $\Lambda_0 := \{\bm{\beta} \;|\; C_{\bm{\beta}} \neq 0\} \subset \mathbb{Z}_{\geqslant 0}^t(M)$.
  For every field automorphism $\tau \colon \overline{E} \to \overline{E}$ such that $\tau|_F = \operatorname{id}$ and every $\bm{\beta}_1 \in \Lambda_0$, there exists $\bm{\beta}_2 \in \Lambda_0$ such that $\tau (\sigma^{\bm{\beta}_1} a_1) = \sigma^{\bm{\beta}_2} a_1$.
  \end{step}
  \noindent
  Let $\tau$ act on $\overline{E}[\![x_1, \ldots, x_{n}]\!]$ coefficient-wise.
  Applying $\tau$ to~\eqref{eq:dependence}, we obtain
  \[
  \left[ \sum\limits_{\bm{\beta} \in \mathbb{Z}_{\geqslant 0}^t(M)} C_{\bm{\beta}} e^{ (\sigma^{\bm{\beta}}a_1) x_1 + \ldots + (\sigma^{\bm{\beta}}a_{n}) x_{n}} \right]_{N} = f = \left[ \sum\limits_{\bm{\beta} \in \mathbb{Z}_{\geqslant 0}^t(M)} \tau(C_{\bm{\beta}}) e^{\tau(\sigma^{\bm{\beta}}a_1) x_1 + \ldots + \tau(\sigma^{\bm{\beta}}a_{n}) x_{n}} \right]_{N}
  \]
  Let 
  \[
    S := \sum\limits_{\bm{\beta} \in \mathbb{Z}_{\geqslant 0}^t(M)} C_{\bm{\beta}} e^{(\sigma^{\bm{\beta}}a_1) x_1 + \ldots + (\sigma^{\bm{\beta}}a_{n}) x_{n}} - \sum\limits_{\bm{\beta} \in \mathbb{Z}_{\geqslant 0}^t(M)} \tau(C_{\bm{\beta}}) e^{\tau(\sigma^{\bm{\beta}}a_1) x_1 + \ldots + \tau(\sigma^{\bm{\beta}}a_{n}) x_{n}}.
  \]
  Since, for every $\bm{\beta} \in \mathbb{Z}_{\geqslant 0}^t(M)$, the total degree of $C_{\bm{\beta}}$ does not exceed $M$, we have
  \[
  D_i S = 0, \text{ where } D_i := \prod\limits_{\bm{\beta} \in \mathbb{Z}_{\geqslant 0}^t(M)} (\partial_i - \sigma^{\bm{\beta}} a_i)^{M + 1} \cdot \prod\limits_{\bm{\beta} \in \mathbb{Z}_{\geqslant 0}^t(M)} (\partial_i - \tau(\sigma^{\bm{\beta}} a_i))^{M + 1}
  \]
  for every $1 \leqslant i \leqslant {n}$. 
  Since the order of $D_i$ is does not exceed $2(M + 1) \cdot |\mathbb{Z}_{\geqslant 0}^t(M)| \leqslant 2(M + 1)^{t + 1} = N$ and $[S]_{N} = 0$, Lemma~\ref{lem:operator} implies that $S = 0$.
  
  Since $a_1$ is $2M$-nonperiodic, the set $\{ \sigma^{\bm{\beta}}a_1 \;|\; \bm{\beta} \in \Lambda_0 \}$ contains $|\Lambda_0|$ distinct elements.
  If the number of distinct elements in the set 
  \[
  \{ \sigma^{\bm{\beta}}a_1 \;|\; \bm{\beta} \in \Lambda_0 \} \cup \{ \tau(\sigma^{\bm{\beta}}a_1) \;|\; \bm{\beta} \in \Lambda_0 \}
  \]
  is greater than $|\Lambda_0|$, then there is $\bm{\beta}_{0} \in \Lambda_0$ such that
 \[
    \tau\sigma^{\bm{\beta}_0} a_1 \not\in \{ \sigma^{\bm{\beta}}a_1 \;|\; \bm{\beta} \in \Lambda_0 \} \cup \{ \tau(\sigma^{\bm{\beta}}a_1) \;|\; \bm{\beta} \in \Lambda_0, \; \bm{\beta} \neq \bm{\beta}_0 \}.
  \]
  Then the equation $S = 0$ implies that the exponential power series $e^{\tau(\sigma^{\bm{\beta}_0} a_1) x_1 + \ldots + \tau(\sigma^{\bm{\beta}_0} a_{n}) x_{n}}$ can be written as a $E(x_1, \ldots, x_{n})$-linear combination of exponential power series with the exponents different from $\tau(\sigma^{\bm{\beta}_0} a_1) x_1 + \ldots + \tau(\sigma^{\bm{\beta}_0} a_{n}) x_{n}$, and this is impossible.
  Thus, for every $\bm{\beta}_0 \in \Lambda_0$, $\tau(\sigma^{\bm{\beta}_0} a_1) \in \{ \sigma^{\bm{\beta}}a_1 \;|\; \bm{\beta} \in \Lambda_0 \}$.
  
  \begin{step}
  $a_1 \in F$.
  \end{step}
  \noindent
  Consider $\bm{\beta}_0 \in \Lambda_0$.
  Step~\ref{step:2_general} implies that every element conjugate to $\sigma^{\bm{\beta}_0} a_1$ in $\overline{E}$ over $F$ is of the form $\sigma^{\bm{\beta}} a_1$, where $\bm{\beta} \in \Lambda_0$.
  In particular, $\sigma^{\bm{\beta}_0} a_1$ is algebraic over $F$. 
  Consider the minimal polynomial $P(t) \in F[t]$ for $\sigma^{\bm{\beta}_0} a_1$ over $F$.
  The roots of $P(t)$ form a subset in $\{ \sigma^{\bm{\beta}}a_1 \mid \bm{\beta} \in \Lambda_0 \}$.
  We define $\Lambda_P := \{ \bm{\beta} \in \Lambda_0 \mid P(\sigma^{\bm{\beta}}a_1) = 0\}$.
  Let $\bm{\beta}_1$ and $\bm{\beta}_2$ be the smallest and the largest elements of $\Lambda_P$ with respect to the lexicographic ordering, respectively.
  Let
  \[
  Q(t) := \sigma^{\bm{\beta}_2 - \bm{\beta}_1}P(t) \in F[t].
  \]
  Then the set of roots of $Q$ in $\overline{E}$ is exactly $\{ \sigma^{\bm{\beta}} a_1 \mid \bm{\beta} \in \bm{\beta}_2 - \bm{\beta}_1 + \Lambda_P\}$.
  We will show that 
  \begin{equation}\label{eq:intersection}
  \Lambda_P \cap (\bm{\beta}_2 - \bm{\beta}_1 + \Lambda_P) = \{\bm{\beta}_2\}.
  \end{equation}
  Assume that there is an element $\bm{\beta}_3$ in the intersection such that $\bm{\beta}_3 \neq \bm{\beta}_2$.
  Then $\bm{\beta}_3 - \bm{\beta}_2 + \bm{\beta}_1 \in \Lambda_P$.
  The maximality of $\bm{\beta}_2$ implies that $\bm{\beta}_3 - \bm{\beta}_2 <_{lex} \bm{0}$.
  Then $\bm{\beta}_3 - \bm{\beta}_2 + \bm{\beta}_1 <_{lex} \bm{\beta}_1$, and this contradicts the minimality of $\bm{\beta}_1$ and proves~\eqref{eq:intersection}.
  
  Consider any common root of $P$ and $Q$.
  This root can be written as $\sigma^{\bm{\beta}_3} a_1$ where $\bm{\beta}_3 \in \Lambda_P$ and as $\sigma^{\bm{\beta}_4} a_1$ where $\bm{\beta}_4 \in \bm{\beta}_2 - \bm{\beta}_1 + \Lambda_P$.
  Then $\sigma^{\bm{\beta}_3 + \bm{\beta}_1} a_1 = \sigma^{\bm{\beta}_4 + \bm{\beta}_1} a_1$.
  Since $|\bm{\beta}_3 + \bm{\beta}_1| \leqslant 2M$, $|\bm{\beta}_4 + \bm{\beta}_1| \leqslant 2M$, and the $a_1$ is $2M$-nonperiodic, we obtain $\bm{\beta}_3 = \bm{\beta}_4$.
  Using~\eqref{eq:intersection}, we see that $\bm{\beta}_2 = \bm{\beta}_3 = \bm{\beta}_4$, so the only common root of $P$ and $Q$ is $\sigma^{\bm{\beta}_2} a_1$.
  Then $\sigma^{\bm{\beta}_2} a_1$ is the only root of $\gcd(P, Q) \in F[t]$, so $\sigma^{\bm{\beta}_2} a_1 \in F$.
  Applying $\sigma^{- \bm{\beta}_2}$, we obtain $a_1 \in F$.
\end{proof}

\subsection{Proof of Theorem~\ref{thm:main}}

\begin{proof}[Proof of Theorem~\ref{thm:main}]
  Lemma~\ref{lem:nondegenerate_field_has_nonzero_Jac} implies that there are elements $c_1, \ldots, c_s \in E$ such that $\det J(c_1, \ldots, c_s) \neq 0$.
  Adding these elements to the set $a_1, \ldots, a_n$ of generators of $E$ over $F$ if necessary, we will further assume that ${\rank} J(a_1, \ldots, a_{n}) {= s}$.
  
  For $f_1, \ldots, f_r \in E$ and a positive integer $m$, we introduce
  \[
  \trdeg (f_1, \ldots, f_r; m) := \trdeg_F F\left( \delta^{\bm{\alpha}}\sigma^{\bm{\beta}}f_i \mid 1 \leqslant i \leqslant r,\, \bm{\alpha} \in \mathbb{Z}_{\geqslant 0}^s(m),\, \bm{\beta} \in \mathbb{Z}_{\geqslant 0}^t(m) \right).
  \]
  Since every element of $E$ is $\DS$-algebraic over $F$, \cite[{Theorem~2.1}]{Levin} implies that there exists a polynomial $p(z) \in \mathbb{Q}[z]$ of degree less than $s + t$ such that 
  \[
    \trdeg(a_1, \ldots, a_n; m) < p(m) \quad\text{ for every } m \in \mathbb{Z}_{\geqslant 0}.
  \]
  Since
  \[
  \left\lvert\{ (\bm{\alpha}, \bm{\beta}) \mid \bm{\alpha} \in \mathbb{Z}_{\geqslant 0}^s(m),\, \bm{\beta} \in \mathbb{Z}_{\geqslant 0}^t(m) \}\right\rvert = \binom{m + s}{s} \binom{m + t}{t}
  \]
  and $\binom{m + s}{s} \binom{m + t}{t}$ is a polynomial of degree $s + t$ in $m$, there exists $M \in \mathbb{Z}_{\geqslant 0}$ such that 
  \[
  \trdeg(a_1, \ldots, a_n; M + n) = p(M + n) < \left\lvert\{ (\bm{\alpha}, \bm{\beta}) \mid \bm{\alpha} \in \mathbb{Z}_{\geqslant 0}^s(M),\, \bm{\beta} \in \mathbb{Z}_{\geqslant 0}^t(M)\}\right\rvert.
  \]
  We will prove by induction on $\ell$ that, for every $0 \leqslant \ell \leqslant n$, there exists $b_\ell \in E$ such that
  \begin{enumerate}[label=\textbf{(R\arabic*)}]
      \item\label{cond:nonperiodic} $b_\ell$ is $2M$-nonperiodic;
      \item\label{cond:Jac} $\det J(\delta_1 b_\ell, \ldots, \delta_s b_\ell) \neq 0$;
      \item\label{cond:generator} $E = F\langle b_\ell, a_{\ell + 1}, \ldots, a_n\rangle$;
      \item\label{cond:trdeg} $\trdeg(b_\ell, a_{\ell + 1}, \ldots, a_n; M + n - \ell) < D_M :=  \left\lvert\{ (\bm{\alpha}, \bm{\beta}) \mid \bm{\alpha} \in \mathbb{Z}_{\geqslant 0}^s(M),\, \bm{\beta} \in \mathbb{Z}_{\geqslant 0}^t(M)\}\right\rvert$.
  \end{enumerate}
  Since $E = F\langle b_n\rangle$, proving the existence of such $b_0, \ldots, b_n$ will prove the theorem.
  
  Consider the base case $\ell = 0$.
  Lemma~\ref{lem:existence_nonperiodic_and_nonzeroJac}  implies that there exists a polynomial $P \in F[x_1, \ldots, x_n]$ such that $b_0 := P(a_1, \ldots, a_n)$ is $2M$-nonperiodic and $\det J(\delta_1 b_0, \ldots, \delta_s b_0) \neq 0$.
  Thus, $b_0$ satisfies~\ref{cond:nonperiodic} and~\ref{cond:Jac}.
  \ref{cond:generator} is trivially satisfied.
  Finally, since $b_0 \in F(a_1, \ldots, a_n)$, we have
  \[
  \trdeg(b_0, a_1, \ldots, a_n; M + n) = \trdeg(a_1, \ldots, a_n; M + n),
  \]
  so~\ref{cond:trdeg} also holds.
  
  Assume that we have constructed $b_\ell$ for some $\ell \geqslant 0$.
  We set $N := 2(M + 1)^{t + 1}$ (as in Lemma~\ref{lem:core}) and $\Gamma := \{0, \ldots, N\}^{s + 1} \subset \mathbb{Z}^{s + 1}$, introduce a set variables 
  \[
    \Theta := \{ \theta_{\bm{\gamma}} \;|\; \bm{\gamma} \in \Gamma \} \cup \{\theta_{-1}\},
  \]
  and extend the action of $\Delta \cup \Sigma$ from $E$ to $E(\Theta)$ by making all elements of $\Theta$ to be $\DS$-constants.
  Let 
  \begin{equation}\label{eq:primitive_general}
  B_{\ell + 1} := \theta_{-1}a_{\ell + 1} + \sum\limits_{\bm{\gamma} \in \Gamma} \frac{\theta_{\bm{\gamma}}}{\bm{\gamma}!} b_\ell^{\gamma_{0}} (\delta_{1} b_\ell)^{\gamma_{1}} \ldots (\delta_s b_\ell)^{\gamma_s}.
  \end{equation}
  We regard any point $\varphi \in \mathbb{Q}^{|\Theta|}$ as a function $\varphi\colon \Theta \to \mathbb{Q}$ and extend it to a $E$-algebra homomorphism $\varphi\colon E[\Theta] \to E$.
  
  \begin{claim}
  There exists a Zariski open nonempty subset $U_1  \subset \mathbb{Q}^{|\Theta|}$ such that 
  \[
  F\langle \varphi(B_{\ell + 1}) \rangle = F\langle b_{\ell}, a_{\ell + 1} \rangle \quad \text{ for every }\quad \varphi \in U_1.
  \]
  \end{claim}
  \noindent
  Since
  \begin{align*}
  &\trdeg_{F(\Theta)} F(\Theta, \{ \delta^{\bm{\alpha}}\sigma^{\bm{\beta}} B_{\ell + 1} \;|\; \bm{\alpha} \in \mathbb{Z}_{\geqslant 0}^s(M),\, \bm{\beta} \in \mathbb{Z}_{\geqslant 0}^t(M) \} ) \leqslant \trdeg(a_{\ell + 1}, b_\ell, \delta_2 b_\ell, \ldots, \delta_s b_\ell; M) \leqslant \\
  &\leqslant \trdeg(a_{\ell + 1}, b_\ell; M + 1) \leqslant \trdeg(b_\ell, a_{\ell + 1}, \ldots, a_n; M + 1) < D_M,
  \end{align*}
  the elements of $\{ \delta^{\bm{\alpha}}\sigma^{\bm{\beta}} B_{\ell + 1} \;|\; \bm{\alpha} \in \mathbb{Z}_{\geqslant 0}^s(M),\, \bm{\beta} \in \mathbb{Z}_{\geqslant 0}^t(M) \}$ are algebraically dependent over $F(\Theta)$.
  Thus, there exists a $\DS$-polynomial $R \in F(\Theta)[ \delta^{\bm{\alpha}} \sigma^{\bm{\beta}} z \mid \bm{\alpha} \in \mathbb{Z}_{\geqslant 0}^s(M),\, \bm{\beta} \in \mathbb{Z}_{\geqslant 0}^t(M) ]$ such that $R(B_{\ell + 1}) = 0$.
  We will assume that $R$ is chosen to be of the minimal possible total degree.
  We introduce
  \[
  R_{\bm{\alpha}, \bm{\beta}} := \frac{\partial R}{\partial \delta^{\bm{\alpha}} \sigma^{\bm{\beta}} z} (B_{\ell + 1}) \text{ for } \bm{\alpha} \in \mathbb{Z}_{\geqslant 0}^s(M),\, \bm{\beta} \in \mathbb{Z}_{\geqslant 0}^t(M) \quad\text{ and }\quad R_{\theta_{\bm{\gamma}}} := \frac{\partial R}{\partial \theta_{\bm{\gamma}}}(B_{\ell + 1}) \text{ for } \bm{\gamma} \in \Gamma.
  \]
  The minimality of the degree of $R$ implies that not all of $R_{\bm{\alpha}, \bm{\beta}}$ are zero.
  Consider any $\bm{\gamma} \in \Gamma$.
  Differentiating $R(B_\ell) = 0$ with respect to $\theta_{\bm{\gamma}}$, we obtain
  \begin{equation}\label{eq:derivative}
  \sum\limits_{\substack{\bm{\alpha} \in \mathbb{Z}_{\geqslant 0}^s(M),\\ \bm{\beta} \in \mathbb{Z}_{\geqslant 0}^t(M)}} R_{\bm{\alpha}, \bm{\beta}} \delta^{\bm{\alpha}} \sigma^{\bm{\beta}} \left( \frac{{b_\ell^{\gamma_0}} (\delta_{1} b_\ell)^{\gamma_{1}} \ldots (\delta_s b_\ell)^{\gamma_s}}{\bm{\gamma}!} \right) = -R_{\theta_{\bm{\gamma}}}.
  \end{equation}
  We extend the action of $\Delta \cup \Sigma$ from $E(\Theta)$ to $E(\Theta)[\![x_1, \ldots, x_s]\!]$ as in~\eqref{eq:extend_der_gen}.
  We multiply~\eqref{eq:derivative} by $x_{0}^{\gamma_{0}}\ldots x_s^{\gamma_s}$ and sum such equations over all $\bm{\gamma} \in \Gamma$.
  We obtain
  \begin{equation}\label{eq:with_exponents}
    \sum\limits_{\substack{\bm{\alpha} \in \mathbb{Z}_{\geqslant 0}^s(M),\\ \bm{\beta} \in \mathbb{Z}_{\geqslant 0}^t(M)}} R_{\bm{\alpha}, \bm{\beta}} \left[ \delta^{\bm{\alpha}} \sigma^{\bm{\beta}} e^{b_\ell x_{0} + (\delta_{1} b_\ell) x_{1} + \ldots + (\delta_s b_\ell) x_s} \right]_{N} = \sum\limits_{\bm{\gamma} \in {\Gamma}} -R_{\theta_{\bm{\gamma}}}x_{0}^{\gamma_{0}} \ldots x_s^{\gamma_s}.
  \end{equation}
  We apply Lemma~\ref{lem:core} to~\eqref{eq:with_exponents} with $a_1 = b_\ell, a_2 = \delta_{1} b_\ell, \ldots, a_{s + 1} = \delta_s b_\ell$ and $F = F\langle \Theta, B_{\ell + 1} \rangle$, and deduce that $b_\ell \in F\langle \Theta, B_{\ell + 1} \rangle$.
  Then there exist nonzero $\DS$ polynomials $P_1, P_2 \in F[\Theta]\{z\}$ such that 
  \begin{equation}\label{eq:formula_b_ell}
  b_\ell = \frac{P_1(B_{\ell + 1})}{P_2(B_{\ell + 1})}.
  \end{equation}
  We define $U_1 := \{\varphi \in \mathbb{Q}^{|\Theta|} \;|\; \varphi(P_2(B_{\ell + 1})) \neq 0 \text{ and } \varphi(\theta_{-1}) \neq 0\}$.
  Since $P_2$ is a nonzero polynomial, $U_1$ is nonempty.
  Consider  $\varphi \in U_1$. Then \eqref{eq:formula_b_ell} implies that $b_\ell \in F\langle \varphi(B_{\ell+ 1}) \rangle$.
  Since $\varphi(\theta_{-1}) \neq 0$, $a_{\ell + 1} \in F\langle \varphi(B_{\ell+ 1}) \rangle$, so $F\langle \varphi(B_{\ell + 1}) \rangle = F\langle b_{\ell}, a_{\ell + 1} \rangle$.
  The claim is proved.
  
  \begin{claim}
  Let
  \[
  U_2 := \{ \varphi \in \mathbb{Q}^{|\Theta|} \;|\; \det J(\delta_1 \varphi(B_{\ell + 1}), \ldots, \delta_s \varphi(B_{\ell + 1})) \neq 0 \text{ and } \varphi(B_{\ell + 1}) \text{ is $2M$-nonperiodic}\}.
  \]
  Then $U_2$ is a nonempty Zariski open set.
  \end{claim}
  \noindent
  The fact that $\varphi(B_{\ell + 1})$ is $2M$-nonperiodic can be expressed by a system of inequations as in the proof of Lemma~\ref{lem:existence_nonperiodic_and_nonzeroJac}.
  Thus, $U_2$ is defined by a system of inequations, so it is open.
  Consider $\varphi_0 \in \mathbb{Q}^{|\Theta|}$ defined by
  \[
  \varphi_0(\theta_{\bm{\gamma}}) = \begin{cases} 1, \text{ if } \bm{\gamma} = (1, 0, \ldots, 0),\\
  0, \text{ otherwise},\end{cases} \quad \varphi_0(\theta_{-1}) = 0.
  \]
  Then $\varphi_0(B_{\ell + 1}) = b_\ell$.
  We have $\det J(\delta_1 b_\ell, \ldots, \delta_s b_\ell) \neq 0$ due to~\ref{cond:Jac} and $b_\ell$ is $2M$-nonperiodic due to~\ref{cond:nonperiodic}.
  Thus, $\varphi_0 \in U_2$, so $U_2 \neq \varnothing$.
  The claim is proved.
  
  Consider $\varphi \in U_1 \cap U_2$ and define $b_{\ell + 1} := \varphi(B_{\ell + 1})$.
  Then~\ref{cond:generator} holds because $\varphi \in U_1$, \ref{cond:nonperiodic} and~\ref{cond:Jac} hold because $\varphi \in U_2$.
  Since $b_{\ell + 1} \in F(a_{\ell + 1}, b_\ell, \delta_1b_\ell, \ldots, \delta_s b_\ell)$, we have
  \[
  \trdeg(b_{\ell + 1}, a_{\ell + 2}, \ldots, a_n; M + n - \ell - 1) \leqslant \trdeg(b_\ell, a_{\ell + 1}, \ldots, a_n; M + n - \ell) < D_M.
  \]
  This proves~\ref{cond:trdeg} for $b_{\ell + 1}$ and finishes the proof of the existence of $b_0, \ldots, b_n$ such that \ref{cond:nonperiodic}, \ref{cond:Jac}, \ref{cond:generator}, and~\ref{cond:trdeg} hold.
\end{proof}

\section*{Acknowldegements}

The author is grateful to Lei Fu, Alexey Ovchinnikov, Thomas Scanlon, and the referee for their suggestions and helpful discussions.
This work has been partially supported by NSF grants CCF-1564132, CCF-1563942, DMS-1853482, DMS-1853650, and DMS-1760448, by
PSC-CUNY grants \#69827-0047 and \#60098-0048.

\bibliographystyle{abbrvnat}
\bibliography{bibdata}

\end{document}